\documentclass[11pt,notitlepage,twoside,a4paper]{amsart}
 \usepackage{amsfonts}
 \usepackage{mathrsfs}

\usepackage{amsmath,amssymb,enumerate}

\usepackage{epsfig,fancyhdr,color}

\usepackage{amssymb}
\usepackage{amsmath,amsthm}
\usepackage{latexsym}
\usepackage{amscd}
\usepackage{psfrag}
\usepackage{graphicx}
\usepackage[latin1]{inputenc}
\usepackage[all]{xy}

\newtheorem{theorem}{\rm\bf Theorem}[section]
\newtheorem{proposition}[theorem]{\rm\bf Proposition}
\newtheorem{lemma}[theorem]{\rm\bf Lemma}
\newtheorem{corollary}[theorem]{\rm\bf Corollary}
\newtheorem*{theorem 1}{\rm\bf Proposition 1}
\newtheorem*{theorem 2}{\rm\bf Proposition 2}
\newtheorem{question}{\rm\bf Question}
\theoremstyle{definition}

\theoremstyle{remark}
\newtheorem{remark}[theorem]{\rm\bf Remark}

\newcommand{\ML}{\mathcal{ML}}
\newcommand{\MF}{\mathcal{MF}}

\newcommand{\PMF}{\mathcal{PMF}}
\newcommand{\n}{n \to \infty}

\newcommand{\R}{\mathbb{R}}

\def\interieur#1{\mathord{\mathop{\kern 0pt #1}\limits^\circ}}

\title[Convergence of earthquake and horocycle paths]{Convergence of earthquake and horocycle paths to the boundary of Teichm\"uller space}

\author{Manman Jiang}
\address{Manman Jiang, Department of Mathematics, Sun Yat-sen University, 510275, Guangzhou, P. R. China}
\email{jiangmanm@126.com}

\author{Weixu Su}
\address{Weixu Su, School of Mathematics, Fudan University, 200433, Shanghai, P. R. China and Shanghai Center
for Mathematical Sciences (SCMS),  200433, Shanghai, P. R. China}
\email{suwx@fudan.edu.cn}

\date{\today}

\thanks{M. Jiang is partially supported by NSFC No: 11271378. W. Su is partially supported by NSFC No: 11201078 }

\begin{document}

\begin{abstract}
We study the convergence of earthquake paths and horocycle paths in the Gardiner-Masur compactification of Teichm\"uller space. 
We show that an earthquake path directed by a uniquely ergodic or simple closed measured geodesic lamination
converges to the Gardiner-Masur boundary.
 Using the embedding of flat metrics into the space of geodesic currents, we prove that a horocycle path in Teichm\"uller space, induced by 
 a quadratic differential whose vertical measured foliation is unique ergodic, converges to the Gardiner-Masur boundary and to the Thurston boundary.
\end{abstract}

\maketitle


\medskip

\medskip

\maketitle
\section{Introduction}
Let $M$ be an  oriented surface of genus $g$ with $n$ punctures.
In this paper, we always assume that $3g-3+n>0$.
Let $\mathcal{T}(M)$ be the Teichm\"uller space of marked conformal (hyperbolic) structures on $M$.
There are several compactifications of $\mathcal{T}(M)$, such as the Thurston compactification and the Bers compactification,
which are extensively used in the study of degeneration of hyperbolic structures and the action of mapping class group.
 In \cite{GM}, Gardiner and Masur defined a compactification  of $\mathcal{T}(M)$, namely the\emph{ Gardiner-Masur compactification}, and showed that the Thurston boundary is strictly contained in the Gardiner-Masur boundary.

The convergence of certain natural rays in Teichm\"uller space to the above mentioned boundaries
 has attracted a lot of attentions
(We say that a ray in Teichm\"uller space
\emph{converges} if  the limit set of the ray in certain compatification is
a unique point).
The following is a  collection of some (not all) known results:
 \begin{itemize}
 \item[$\bullet$]\emph{ Convergence of Teichm\"uller rays}. Every Teichm\"uller geodesic ray converges to a unique point on the Gardiner-Masur boundary \cite{LS}.
 Typical Teichm\"uller geodesic rays, which are one-cylinder Strebel or uniquely ergodic, converge to the Thurston boundary \cite{Mas1}.
Examples of Teichm\"uller rays that diverge in the Thurston boundary were first constructed by Lenzhen \cite{Len}.
  Moreover, there exists Teichm\"uller geodesic ray whose limit set on the Thurston boundary is a line segment, see \cite{CMW}.

  \item[$\bullet$] \emph{Convergence of stretch rays}. Stretch ray are geodesic rays of the Thurston asymmetric metric \cite{Thurston}.
  Papadopoulos \cite{Pap} proved that a stretch ray directed by a complete measured geodesic lamination converges to a unique point on the Thurston boundary. A stronger result of Walsh \cite{Walsh} implies that any geodesic ray of the Thurston asymmetric metric (thus any stretch ray)
      converges to  a unique point on the Thurston boundary.

  \item[$\bullet$] \emph{Convergence of earthquake rays}. Any earthquake path determined by a
  measured geodesic lamination $\mu$ converges to the projective class  of $\mu$ on the  Thurston boundary
  (This can be proved by using Kerckhoff's variational formula of hyperbolic length functions \cite{Ker1}).
 \end{itemize}

  In the first part of this paper, we initiate the study of the following question:
\begin{question}
Let $\mu$ be  a measured geodesic lamination and let $\{E^{t\mu}(X)\}_{t\geq 0}$
 be the earthquake path determined by $\mu$. Determine the limit of $\{E^{t\mu}(X)\}_{t\geq 0}$ on the Gardiner-Masur boundary as $t$
 tends to $\infty$.
\end{question}

We don't know whether the limit is unique for every measured geodesic lamination.
Previous  work of Miyachi \cite{Miy1, Miy2} indicates that the topology of Gardiner-Masur compactification
is more involved.
We will prove the following (see \S \ref{sec:earthquake}):
\begin{theorem}\label{sim & uniq}
If $\mu$ is either uniquely ergodic or a weighted simple closed geodesic, then $\{E^{t\mu}(X)\}_{t\geq 0}$ converges to the
projective class of $\mu$ as a limit point on the Gardiner-Masur boundary.
\end{theorem}

In \cite{Mir}, Mirzakhani showed that the earthquake flow and the  horocylic flow are measurably isomorphic.
Thus it is interesting to consider the  convergence of a horocycle path.
In the second part of this paper, we obtain the following  (see \S \ref{sec:horocycle}):
 \begin{theorem}\label{thm:horo}
Let $q$ be a holomorphic quadratic differential on $X\in \mathcal{T}(M)$ and $\{\pi(h^t(q))\}_{t\in \R}$ be the corresponding horocycle path in $\mathcal{T}(M)$. If the vertical measured foliation $V_q$ of $q$ is uniquely ergodic, then  $\{\pi(h^t(q))\}_{t\in \R}$ converges to
the projective class of $V_q$ on the Gardiner-Masur boundary.
 \end{theorem}
In fact, we will show that when $V_q$  is uniquely ergodic, the horocycle path $\{\pi(h^t(q))\}_{t\in \R}$ also converges to
the projective class of $V_q$ on the Thurston boundary. Since almost all measured geodesic laminations are uniquely ergodic,
we obtain:
\begin{corollary}
Almost every earthquake path or horocycle path in Teichm\"uller space converges to
a unique limit on the Gardiner-Masur (Thurston) boundary.
\end{corollary}
The convergence of a general horocycle path to the Thurston boundary remains an
open question.

\subsection{Acknowledgements}
        The authors would like to thank the referee for the valuable suggestions that helped to improve the paper.
        We are grateful to Athanase Papadopoulos for  careful reading of the manuscript.

\subsection{Note added in proof} Just after this paper was submitted, Alberge posted a paper \cite{Alberge}
which contains another proof of Theorem \ref{thm:horo}. His proof uses Miyachi's extension of Gromov product
to the Gardiner-Masur compactification,
while our proof uses the embedding of flat metrics into the space of geodesic currents, developed by Duchin, Leininger and Rafi.

\section{Preliminaries}\label{Preliminaries}
\subsection{ Teichm\"{u}ller space and extremal length}
Let $M$ be an  oriented surface of genus $g$ with $n$ punctures with $3g-3+n>0$.
The Teichm\"{u}ller space $\mathcal{T}(M)$ is the space of equivalence classes of marked Riemann surfaces
(we shall consider only surfaces of analytically finite type).
Recall that a \emph{marked Riemann surface} is a pair $(X,f)$ where $X$ is a Riemann surface
homeomorphic to $M$,
and $f:M\rightarrow X$ is an orientation-preserving homeomorphism.
Two marked Riemann surfaces $(X_1,f_1),(X_2,f_2)$ are \emph{equivalent} if  there exists a conformal map $g:X_1\rightarrow X_2$ which is homotopic to $f_2\circ f_1^{-1}$.

The \emph{Teichm\"{u}ller distance} between two equivalence classes of  $(X_1,f_1)$ and $(X_2,f_2)$ is defined as following:
 $$d_T\left((X_1,f_1),(X_2,f_2)\right)=\frac 1 2 \inf_f \log K[f],$$
 where the infimum runs over all quasiconformal mappings $f$ in the homotopy class of $f_2\circ f_1^{-1}$ and $K[f]$ is the maximal dilatation of $f$.
  In the following, we will denote a point in $\mathcal{T}(M)$ by $X$ for simplicity,
  without explicit reference to the marking or to the equivalence relation.

 A \emph{conformal metric} $\sigma$ on a Riemann surface $X$ is a metric of the form $\sigma(z)|dz|$ in local conformal coordinates,
  where $\sigma(z)$ is a Borel measurable, non-negative function. The $\sigma$-area of $X$ is defined by
                     $$\mathrm{Area}_\sigma(X)=\int_X \sigma^2(z)|dz|^2,$$
 and the $\sigma$-length of a simple closed curve  $\alpha$  is defined by
           $$L_\sigma(\alpha)=\inf_{\alpha'}\int_{\alpha'}\sigma(z)|dz|,$$
 where the infimum is taken over all simple closed curves $\alpha'$  homotopic to $\alpha$.

 Note that a (holomorphic) quadratic differential $q=q(z)dz^2$ on $X$ induces a natural conformal metric $|q|^{1/2}$.
Usually, the $|q|^{1/2}$-area of $X$ is denoted by $\|q\|$, and the $|q|^{1/2}$-length of a simple closed curve  $\alpha$
 is denoted by $L_q(\alpha)$.

  With the above notation, the \emph{extremal length} of $\alpha$ on $X$ is defined by
  $$\mathrm{Ext}_{\alpha}(X)=\sup_{\rho} \frac{L^2_{\rho}(\alpha)}{\mathrm{Area}_\rho(X)},$$
where the supremum is taken over all conformal metrics on $X$ with finite positive area.
It is not hard to see that the definition of $\mathrm{Ext}_{\alpha}(X)$ only depends on the homotopy class of
$\alpha$ and the equivalence class of $X$ in Teichm\"uller space.

Denoted by  $\mathcal{S}$ the set of homotopy classes of essential simple closed curves (neither homotopic to a point nor to a puncture) on $X$.
In \cite{Ker2}, Kerckhoff gave an interpretation of the Teichm\"{u}ller distance as
$$d_T(X_1,X_2)=\frac{1}{2}\log\sup_{\alpha\in \mathcal{S}} \frac{\mathrm{Ext}_\alpha(X_1)}{\mathrm{Ext}_\alpha(X_2)}.$$

By the Uniformization Theorem,
a Riemann surface  inherits a unique complete hyperbolic metric of finite area in its conformal class.
As a result, the Teichm\"uller space $\mathcal{T}(M)$ can be considered as the space of equivalence classes of hyperbolic metrics on $M$,
 where two hyperbolic metrics are equivalent if there is an isometry between them which is homotopic to the identity map on $M$.

Given $X\in \mathcal{T}(M)$ and $\alpha\in \mathcal{S}$, we denote by $\ell_{\alpha}(X)$
 the hyperbolic length  of $\alpha$ on $X$, defined as the length of the simple closed geodesic in the homotopy class of $\alpha$.

\subsection{ Measured foliations and Thurston compactification}
We briefly recall the Thurston theory of measured foliations. For details we refer to \cite{FLP}.

A \emph{measured foliation} $(F,\mu)$ on $M$ is a foliation
(with singularities of the type $n$-prongs, $n\geq 3$) $F$ endowed with a transversely invariant measure $\mu$.
For simplicity we just denote the pair $(F,\mu)$ by $F$.
The intersection number $i(F,\alpha)$ of $F$ with a simple closed curve $\alpha$ is defined by
\[ i(F,\alpha)\triangleq \inf_{\alpha'}\int_{\alpha'}d\mu,\]
where $\alpha' $ ranges over all simple closed curves homotopic to $\alpha$.
Two measured foliations $F_1,\ F_2$ are called \textit{equivalent} if $i(F_1,\alpha)=i(F_2, \alpha)$ for all  $\alpha\in \mathcal{S}$.
We denote by $\mathcal{MF}$ the space of equivalent classes of measured foliations on $M$.
It was shown by Thurston that $\mathcal{MF}$ is homeomorphic to the Euclidean space of dimension
$6g-g+2n$.

Let $\R^{\mathcal{S}}_{+}$ be the space of non-negative functionals on $\mathcal{S}$.
We endow $\mathcal{S}$ with the discrete topology and $\R^{\mathcal{S}}_{+}$ with
the corresponding product topology. The projective space of  $\R^{\mathcal{S}}_{+}$
will be denoted by $\mathcal{P}\R^{\mathcal{S}}_+$.
There is a natural embedding of $\mathcal{MF}$  into $\R^{\mathcal{S}}_+$ defined by

 \begin{eqnarray*}
\mathcal{MF} &\to& \R^{\mathcal{S}}_+ \\
F &\mapsto& \left(i(F,\alpha)\right)_{\alpha\in \mathcal{S}}.
 \end{eqnarray*}
 This induces an embedding from the space of projective classes of measured foliations, denoted by $\mathcal{PMF}$, into $\mathcal{P}\R^{\mathcal{S}}_+$.
 The image of $\mathcal{PMF}$ in  $\mathcal{P}\R^{\mathcal{S}}_+$ is the so called \emph{Thurston boundary}.

There is another natural embedding, from $\mathcal{T}(M)$ to $\R^{\mathcal{S}}_+$, given by
 \begin{eqnarray*}
\mathcal{T}(M)&\to& \R^{\mathcal{S}}_+ \\
X&\mapsto& \left(\ell_\alpha(X)\right)_{\alpha\in \mathcal{S}}.
 \end{eqnarray*}
 It was observed by Thurston that the projection of the above map to $\mathcal{P}\R^{\mathcal{S}}_+ $ is also an embedding. Moreover, the closure of the image of $\mathcal{T}(M)$ in $\mathcal{P}\R^{\mathcal{S}}_+$ is a compact set, called the \textit{Thurston compactification} of $\mathcal{T}(M)$. The boundary of the Thurston compactification is  the Thurston boundary, which,
  as we mentioned above, can be identified with the space of projective measured foliations $\PMF$.

\subsection{Geodesic laminations and earthquake}
We endow $M$ with a hyperbolic metric.
A \emph{geodesic lamination} $\lambda$ is a closed subset of $M$ foliated by complete simple geodesics called \emph{leaves} of the lamination.
A \emph{measured geodesic lamination} is a geodesic lamination $\lambda$
together with a transverse measure $\mu$.
To lighten notation, we shall sometimes talk about a ``measured lamination"
instead of a ``measured geodesic lamination".
We shall denote such a measured lamination by $(\lambda,\mu)$ or, sometimes, $\mu$ for simplicity.
All the measured laminations are assumed to have compact support.
An example of a measured lamination is a weighted simple closed
geodesic, that is,  a simple closed geodesic $\gamma$ equipped with a positive weight $a > 0$.
The measure deposited on a transverse arc $k$ is then the sum of the Dirac masses at the
intersection points between $k$ and $\gamma$ multiplied by the weight $a$.
We shall denote such a measured lamination by $a\gamma$.

 Denote by $\mathcal{ML}$ the space of measured laminations on $M$.
The intersection number $i(\mu,\alpha)$ of $\mu \in \ML$ with $\alpha \in \mathcal{S}$ is defined similar to that of a measured foliation.
Since there is an one-to-one correspondence between $\ML$ and $\MF$ (\cite{Le}), we will use the terms measured foliation and measured lamination alternatively without further explanation.

Next, we briefly explain the notion of \textit{earthquakes} introduced by Thurston (\cite{Ker1,Th}).
 Fix a hyperbolic surface $X$ and a simple closed geodesic $\gamma$ on $X$.
 Cut along $\gamma$ and glue the boundary components back with a left twist of distance $t$, i.e., the two images of any point on $\gamma$ are separated by distance $t$ along the image of $\gamma$.
 Denote the new hyperbolic surface as $E^{t\gamma}(X)$. Note that the notions of ``left" and ``right" twists depend only on the orientation of $X$.
 We define the time-$t$ twist deformation of $X$  along a weighed simple closed geodesic $a\gamma$  to be the new structure obtained from $X$ by twisting left distance $ta$ along $\gamma$.

  For a general measured lamination $\mu$, the \emph{left earthquake deformation} is defined as a limit.
  More precisely, as shown by Thurston, the set of weighted simple closed geodesics $\R_+ \times \mathcal{S}$ is dense in $\mathcal{ML}$.
 For any sequence of weighted simple closed geodesics $\{a_i\gamma_i\}_{i=1}^{\infty}\subset \R_+ \times \mathcal{S}$ that converges to $\mu$,
  $\{E^{ta_i\gamma_i}(X)\}^{\infty}_{i=1}$ converges to a point in $\mathcal{T}(M)$ for any $t\geq 0$ (for a proof see  Kerckhoff (\cite{Ker1})).
  The limit, denoted as $E^{t\mu}(X)$, is independent of the choice of  $\{a_i\gamma_i\}_{i=1}^{\infty}$.
  We call $E^{t\mu}(X)$ the time-$t$ earthquake deformation of $X$ along $\mu$
  and $\{E^{t\mu}(X)\}_{t\geq0}$ the \emph{earthquake path} determined by $\mu$.

\subsection{Gardiner-Masur boundary}
Define a mapping $\Psi$ from $\mathcal{T}(M)$ to  $\mathcal{P}\R_{+}^\mathcal{S} $ by
\begin{eqnarray*}
\Psi:\mathcal{T}(M) &\rightarrow& \mathcal{P}\R_{+}^\mathcal{S} \\
X &\mapsto& [\left(\sqrt{\mathrm{Ext}_{\alpha}(X)}\right)_{\alpha\in \mathcal{S}}]
\end{eqnarray*}
where $[(\sqrt{\mathrm{Ext}_{\alpha}(X)})_{\alpha\in \mathcal{S}}]$ denotes the projective class of $(\sqrt{\mathrm{Ext}_{\alpha}(X)})_{\alpha\in S}$.

Gardiner and Masur \cite{GM} showed that the mapping $\Psi$ is an embedding and the image is relatively compact.
Denote by $cl_{GM}\mathcal{T}(M)$ the closure  of $\Psi(\mathcal{T}(M))$ in $\mathcal{P}\R_{+}^\mathcal{S}$. The boundary $\partial_{GM}\mathcal{T}(M)=cl_{GM}\mathcal{T}(M)-\Psi(\mathcal{T}(M))$ is called the \textit{Gardiner-Masur boundary} of $\mathcal{T}(M)$.

 Miyachi \cite{Miy1} showed that for any $p\in \partial_{GM}\mathcal{T}(M) $,
there is a non-negative continuous function $\varepsilon_p: \mathcal{MF}\rightarrow \R$
such that the projective class of $(\varepsilon_p(\alpha))_{\alpha\in \mathcal{S}}$ determines $p$ as a point of $\partial_{GM}\mathcal{T}(M)$
in the following sense:  any sequence $\{X_n\}_{n=1}^{\infty}$ in $\mathcal{T}(M)$ converging to $p \in \partial_{GM}\mathcal{T}(M)$ satisfies
$$\lim_{n\to \infty}\frac{\sqrt{\mathrm{Ext}_{\mu}(X_n)}}{\sqrt{\mathrm{Ext}_{\nu}(X_n)}}=\frac{\varepsilon_p(\mu)}{\varepsilon_p(\nu)}$$
for all $\mu, \nu\in \mathcal{MF}$ with $\varepsilon_p(\nu)\neq 0$.

The following theorem of  Miyachi will be used in our paper.

 \begin{theorem} [\cite{Miy2}]\label{Miy}
 Let $p\in \partial_{GM}\mathcal{T}(M)) $. Suppose that $G\in \mathcal{MF}$ is either uniquely ergodic or a weighted simple closed curve.
If $\varepsilon_p(F)=0$ for any  $F\in\mathcal{MF} $ with $i(F,G)=0$, then there is a constant  $c>0$ such that $\varepsilon_p(H)=c \cdot i(H,G)$ for any $H\in \mathcal{MF}$ .
\end{theorem}

\section{Earthquake path directed by a simple closed geodesic}\label{sec:earthquake}
 We fix a simple closed geodesic $\alpha$ on a hyperbolic surface $X$
  and let $\mu$ be a measured lamination whose support is  $\alpha$, with weight $k$, where $k>0$.
 Denote by $E^{t\mu}(X)$ the earthquake path in $\mathcal{T}(M)$ starting at $X$ determined by $\mu$.

\begin{lemma}\label{lem:twist}
There exists a constant $C\geq 1$, which only depends on $\mu$ and $X$, such that $$\frac{1}{C}i(\mu,\beta)^{2} \leq  \varliminf_{t\to\infty}\frac{\mathrm{Ext}_{\beta}(E^{t\mu}(X))}{t^{2}}\leq  \varlimsup_{t\to\infty}\frac{ \mathrm{Ext}_{\beta}(E^{t\mu}(X))}{t^{2}}\leq Ci(\mu,\beta)^{2}$$ for any $\beta\in \mathcal{S}$.
\end{lemma}
\begin{proof}

  Since  $E^{t\mu}(X)=E^{tk\alpha}(X)$, by replacing $t$ by $kt$,
  we may assume that $k=1$.

Consider any $\beta\in \mathcal{S }$.
 If $i(\alpha,\beta)=0$, then $\ell_{\beta}(E^{t\alpha}(X))=\ell_{\beta}(X)$ for all $t\geq 0$.
 By Maskit's inequality (\cite{Maskit}), we have

\begin{equation}\label{equ:Maskit}
\frac{\ell_{\beta}(X)^2}{2\pi|\chi(S)|}\leq \mathrm{Ext}_{\beta}(E^{t\alpha}(X))\leq \frac{1}{2}\ell_{\beta}(X)e^{\ell_{\beta}(X)/2}
\end{equation}
for $t\geq 0$. This implies that for any  $\beta\in \mathcal{S}$ satisfying $i(\alpha,\beta)=0$,
$$\frac{\mathrm{Ext}_{\beta}(E^{t\alpha}(X))}{t^2}\rightarrow 0 $$ as $t\rightarrow\infty$.

Next, we consider the case where  $\beta \in \mathcal{S}$ and $i(\alpha,\beta)\neq 0$.
For any $t>0$, we assume that $t=(n+\epsilon)\ell_\alpha(X)$ for some $n\in \mathbb{N}$ and $0\leq\epsilon<1$.
Let $t_n=n\ell_\alpha(X)$ and let $g_n$ be the  $n$-order
 positive Dehn twist around $\alpha$ (We identify $g_n$ as an element of the mapping class group). Then $g_n(X)=E^{t_n\alpha}(X)$ is the $n$-order
 positive Dehn twist of $X$ around $\alpha$.
Since
Dehn twists preserve the Teichm\"uller distance, we have
$$d_T(E^{t_n\alpha}(X),E^{t\alpha}(X))=d_T(X,E^{\epsilon\ell_\alpha\alpha}(X))\leq d.$$

From Kerckhoff's formula, we have
\begin{equation} \label{Ker}
{e^{-2d}}\mathrm{Ext}_{\beta}(E^{t_n\alpha}(X))\leq \mathrm{Ext}_{\beta}(E^{t\alpha}(X))\leq e^{2d}\mathrm{Ext}_{\beta}(E^{t_n\alpha}(X)).
\end{equation}

Note that $$\mathrm{Ext}_{\beta}(E^{t_n\alpha}(X))=\mathrm{Ext}_{\beta}(g_{n}(X))=\mathrm{Ext}_{g_{n}^{-1}(\beta)}(X).$$
One can consider $g_{n}^{-1}(\beta)$ as a sequence of measured laminations (foliations). It follows that (\cite{FLP})

$$\frac{g_{n}^{-1}(\beta)}{n}\rightarrow i(\alpha,\beta)\times\alpha$$   as   $n\rightarrow \infty$.
Here $i(\alpha,\beta)\times\alpha$ denotes the geodesic lamination given by  $\alpha$ with weight $i(\alpha,\beta)$.

In \cite{Ker2}, Kerckhoff showed that the extremal length function is continuous on $\mathcal{MF}$, thus
$$\mathrm{Ext}_{\frac{g_{n}^{-1}(\beta)}{n}}(X)\rightarrow i(\alpha,\beta)^{2}\mathrm{Ext}_{\alpha}(X)$$ as $n\rightarrow \infty$.
As a result
\begin{equation}\label{converge}
\frac{\mathrm{Ext}_{\beta}(E^{t_n\alpha}(X))}{n^2}=\mathrm{Ext}_{\frac{g_{n}^{-1}(\beta)}{n}}(X)\rightarrow \mathrm{Ext}_{\alpha}(X)i(\alpha,\beta)^{2}
\end{equation}
as $n\rightarrow\infty .$

Combining Equation \eqref{equ:Maskit}, \eqref{Ker} and \eqref{converge},
we conclude that there exists a constant $C\geq 1$, which only depends on $\alpha$ and $X$, such that $$\frac{1}{C}i(\alpha,\beta)^{2} \leq  \varliminf_{t\to\infty}\frac{\mathrm{Ext}_{\beta}(E^{t\alpha}(X))}{t^{2}}\leq  \varlimsup_{t\to\infty}\frac{ \mathrm{Ext}_{\beta}(E^{t\alpha}(X))}{t^{2}}\leq Ci(\alpha,\beta)^{2}$$ for any $\beta\in \mathcal{S}$.
\end{proof}
\begin{corollary}\label{coro:simple}
The earthquake path $E^{t\mu}(X)$ converge to $[\alpha]$ (the projective class of $\alpha$) in the sense of the Gardiner-Masur compactification.
\end{corollary}
\begin{proof}  Let $\mathcal{G}$ be the limit set of the earthquake path $\{E^{t\mu}(X)\}_{t\geq0}$ in  the Gardiner-Masur boundary.
We are going to show that $\mathcal{G}$ consists of a single point, the projective class of $\alpha$.

Suppose that $p\in \mathcal{G}$. By definition,  there is a sequence $\{E^{s_j\mu}(X)\}$, $j=1,2,\cdots$,  converging to $p$.
We set $X_j=E^{s_j\mu}(X)$.
We choose  some $\beta_{0}\in \mathcal{S}$ such that $i(\alpha,\beta_{0})\neq0$. Then $\varepsilon_{p}(\beta_{0})\neq0 $
and  we have $$\lim_{j\rightarrow\infty}\frac{\mathrm{Ext}_{\gamma}(X_j)}{\mathrm{Ext}_{\beta_{0}}(X_{j})}\rightarrow \frac{\varepsilon_{p}(\gamma)^{2}}{\varepsilon_{p}(\beta_{0})^{2}}$$ for any $\gamma\in \mathcal{S}$.

Consider any weighted simple closed curve $a\gamma$ with $0\leq i(\alpha,a \gamma)\leq \xi$. It follows from
Lemma \ref{lem:twist} that
$$\frac{\mathrm{Ext}_{a\gamma}(X_{j})}{\mathrm{Ext}_{\beta_{0}}(X_{j})}\leq C^2\frac{i(\alpha,a\gamma)^2}{i(\alpha,\beta_0)^2}\leq \frac{C^2\xi^2}{i(\alpha,\beta_0)^2}.$$
For any $F\in \mathcal{MF}$ with $i(\alpha,F)=0$, since $\R_+\times \mathcal{S}$ is dense in $\MF$,
there are positive numbers $a_n$ and distinct simple closed curves $\gamma_n$ such that $a_n\gamma_n\rightarrow F$  in $\mathcal{MF}$ as $n\rightarrow \infty$ .
Then it follows from the above inequality  that $\varepsilon_{p}(a_n\gamma_n)\rightarrow 0$,
 therefore we have
$\varepsilon_{p}(F)=0$.
By Theorem \ref{Miy}, we  conclude that $p=[\alpha]$ in the Gardiner-Masur boundary.
\end{proof}

\begin{proof}[Proof of Theorem \ref{sim & uniq}]
In the case where $\mu$ is a weighted simple closed curve,
the statement of Theorem \ref{sim & uniq} follows from Corollary \ref{coro:simple}.

Note that for any measured lamination $\mu$,
the earthquake path $E^{t\mu}(X)$ converges to $[\mu]$ under the
Thurston embedding (see Lemma 17 of \S 5 in  \cite{Bon1}). By Corollary 1 of \S 6 in \cite{Miy2}, if $\{Y_n\}_{n=1}^{\infty} $ is a sequence in $\mathcal{T}(M)$ converging to $p\in\partial_{GM}\mathcal{T}(M)$ and if furthermore the sequence converges to $[\mu]$ under the Thurston embedding, then $\mu$ is uniquely ergodic if and only if the zero set of $\varepsilon_p$ is $\{t\mu \ | \ t\geq0\}$.
As a result,
if $\lambda$ is uniquely ergodic (this is equivalent to the fact that the set $\{F\in\mathcal{MF} \ | \  i(F,\mu)=0\}$ is the same as $\{t\mu \ | \ t\geq0\}$), then it follows from Theorem \ref{Miy} that the earthquake path $E^{t\lambda}(X)$ converges to $[\lambda]$ on the Gardiner-Masur boundary.
\end{proof}

It is interesting to study the quasiconvexity of the extremal length functions along an earthquake path.
A function $f:  \mathbb{R} \rightarrow \mathbb{R}_+$ is  called \textit{quasiconvex}
 if there exists some constant $K\geq 1$ such that   $$f(t_2)\leq K \max\{f(t_1),f(t_3)\}$$
 for any $t_1<t_2<t_3$.

\begin{proposition}
Let $\mu$ be a weighted simple closed geodesic.
Then the extremal length of any simple closed curve $\gamma$ is quasiconvex along
the earthquake path $E^{t\mu}(X)$.
\end{proposition}
\begin{proof}
Suppose that $\mu$ is a multiple of $\alpha$.
If $i(\alpha,\gamma)=0$, then from the above discussion that there exists some constant $K_\gamma$, depending only on  $\ell_\gamma(X)$, such that $\frac{1}{K_\gamma}\leq \mathrm{Ext}_\gamma(X_t)\leq K_\gamma$, for any $t\geq0$. It is obvious that $\mathrm{Ext}_\gamma$ is quasiconvex.

If $i(\alpha,\gamma)\neq0$, then there exists some constant $T=T(\gamma,\mu)$ such that for any $t\geq T$,
$$\frac{1}{C}i(\alpha,\gamma)^{2}\leq  \frac{\mathrm{Ext}_{\gamma}(E^{t\mu}(X))}{t^{2}}\leq Ci(\alpha,\gamma)^{2}.$$
Thus if $T\leq t_1\leq t_2$ we get $$\mathrm{Ext}_{\gamma}(E^{t_1 \mu}(X))\leq C i(\alpha,\gamma)^{2}t_1^2\leq C^2\mathrm{Ext}_{\gamma}(E^{t_2 \mu}(X)).$$
Since $\{E^{t\mu}\}_{0\leq t\leq T}$ is a compact set in $\mathcal{T}(M)$,
there must be some constant $K=K(\gamma,\mu)$ such that $\mathrm{Ext}_\gamma$ is $K$-quasicovex.
\end{proof}

\section{Earthquakes directed by multicurves}
In this section,
 we discuss the convergence of earthquake paths directed by rational measured foliations.
The simplest case is the case where  \begin{equation}\label{eq:multi}
\nu=\ell_{\alpha_1}(X)\alpha_1+\ell_{\alpha_2}(X)\alpha_2,
\end{equation}
where $\alpha_1,\alpha_2 \in \mathcal{S}$ with $i(\alpha_1,\alpha_2)=0$.
In this case, $E^{n\nu}(X)$ is the composition of an $n$-order positive Dehn twist around $\alpha_1$ and an $n$-order positive Dehn twist around $\alpha_2$.
Note that since $\alpha_1,\alpha_2$ are disjoint,
it is irrelevant which Dehn twist is performed first.

A simple closed curve $\gamma$ on $M$ is called \emph{separating} if $M\backslash\gamma$ has two connected components.

\begin{proposition}[Necessary condition for convergence]\label{necessary}\label{nec}
Let $\nu$ be a measured lamination as  in \eqref{eq:multi}.
If the earthquake path  $\{X_t=E^{t\nu}(X)\}_{t\geq0}$ converges to a point $p$ on the Gardiner-Masur boundary, then
$$\frac{\mathrm{Ext}_{i(\alpha_1,\gamma)\alpha_1+i(\alpha_2,\gamma)\alpha_2}
(X_t)}{\mathrm{Ext}_{i(\alpha_1,\gamma)\alpha_1+i(\alpha_2,\gamma)\alpha_2}
(X_0)}=\frac{\mathrm{Ext}_{\alpha_1}(X_t)}{\mathrm{Ext}_{\alpha_1}(X_0)}
=\frac{\mathrm{Ext}_{\alpha_2}(X_t)}{\mathrm{Ext}_{\alpha_2}(X_0)},\forall \  \gamma\in \mathcal{S}.$$
Moreover, if both $\alpha_1,\alpha_2$ are separating, or  both $\alpha_1,\alpha_2$ are non-separating,
then the above equation can be refined to
\begin{equation}\label{nec-equ}
 \frac{\mathrm{Ext}_{a_1\alpha_1+a_2\alpha_2}(X_t)}{\mathrm{Ext}_{a_1\alpha_1+a_2\alpha_2}(X_0)}
=\frac{\mathrm{Ext}_{\alpha_1}(X_t)}{\mathrm{Ext}_{\alpha_1}(X_0)}=\frac{\mathrm{Ext}_{\alpha_2}(X_t)}{\mathrm{Ext}_{\alpha_2}(X_0)},\forall \  a_1,a_2>0.
\end{equation}
\end{proposition}
\remark Unfortunately, we are not able to prove or disprove
the above equality. On the other hand,
if we replace extremal length by hyperbolic length, then Equation \eqref{nec-equ}  holds for any disjoint $\alpha_1,\alpha_2$ since
\begin{eqnarray*}
l_{a_1\alpha_1+a_2\alpha_2}(X_t)&=&a_1l_{\alpha_1}(X_t)+a_2l_{\alpha_2}(X_t),\\ l_{\alpha_1}(X_t)&=&l_{\alpha_1}(X_0),\\
l_{\alpha_2}(X_t)&=&l_{\alpha_2}(X_0).
\end{eqnarray*}
\begin{proof}[Proof of Proposition \ref{necessary}]

Let $h_n$ be the composition of an $n$-order positive Dehn twist around $\alpha_1$ and an $n$-order positive Dehn twist around $\alpha_2$.

For any $\gamma \in \mathcal{S}$ we have $$\frac{h_n^{-1}(\gamma)}{n}\rightarrow i(\alpha_1,\gamma)\alpha_1+i(\alpha_2,\gamma)\alpha_2$$ in $\mathcal{MF}$ as $\n.$
Therefore we  get
$$ \frac{\mathrm{Ext}_{\gamma}(X_{n+t})}{n^2}=\frac{\mathrm{Ext}_{\gamma}(h_n(X_t))}{n^2}\rightarrow \mathrm{Ext}_{i(\alpha_1,\gamma)\alpha_1+i(\alpha_2,\gamma)\alpha_2}(X_t)\ \text{as} \ \n.$$

If the earthquake path $\{E^{t\nu}(X)\}$ converges to a point $p$ on the Gardiner-Masur boundary,
 then for any $\gamma_1,\gamma_2\in \mathcal{S}$ with $\varepsilon_p(\gamma_2)\neq0$, we have
$$\frac{\varepsilon_{p}(\gamma_1)}{\varepsilon_{p}(\gamma_2)}=
\lim_{n\rightarrow\infty}\sqrt{\frac{\mathrm{Ext}_{\gamma_1}(X_{n+t})}{\mathrm{Ext}_{\gamma_2}(X_{n+t})}}
=\sqrt{\frac{\mathrm{Ext}_{i(\alpha_1,\gamma_1)\alpha_1+i(\alpha_2,\gamma_1)\alpha_2}
(X_t)}{\mathrm{Ext}_{i(\alpha_1,\gamma_2)\alpha_1+i(\alpha_2,\gamma_2)\alpha_2}
(X_t)}}$$
for $0\leq t<1.$

In particular,
$$\frac{\mathrm{Ext}_{i(\alpha_1,\gamma_1)\alpha_1+i(\alpha_2,\gamma_1)\alpha_2}
(X_t)}{\mathrm{Ext}_{i(\alpha_1,\gamma_2)\alpha_1+i(\alpha_1,\gamma_2)\alpha_2}
(X_t)}=\frac{\mathrm{Ext}_{i(\alpha_1,\gamma_1)\alpha_1+i(\alpha_2,\gamma_1)\alpha_2}
(X_0)}{\mathrm{Ext}_{i(\alpha_1,\gamma_2)\alpha_1+i(\alpha_2,\gamma_2)\alpha_2}
(X_0)},$$
which implies
$$\frac{\mathrm{Ext}_{i(\alpha_1,\gamma_1)\alpha_1+i(\alpha_2,\gamma_1)\alpha_2}
(X_t)}{\mathrm{Ext}_{i(\alpha_1,\gamma_1)\alpha_1+i(\alpha_1,\gamma_1)\alpha_2}
(X_0)}=\frac{\mathrm{Ext}_{i(\alpha_1,\gamma_2)\alpha_1+i(\alpha_2,\gamma_2)\alpha_2}
(X_t)}{\mathrm{Ext}_{i(\alpha_1,\gamma_2)\alpha_1+i(\alpha_2,\gamma_2)\alpha_2}
(X_0)}.$$

Since $\alpha_1,\alpha_2$ are disjoint, there exist $\beta_1, \beta_2\in \mathcal{S}$
such that $i(\beta_i, \alpha_i)>0, i(\beta_i, \alpha_j)=0$ if $i\neq j$.
Replacing $\gamma_2$ in the above equation above by $\beta_1,\beta_2$ respectively, we get
$$\frac{\mathrm{Ext}_{i(\alpha_1,\gamma_1)\alpha_1+i(\alpha_2,\gamma_1)\alpha_2}
(X_t)}{\mathrm{Ext}_{i(\alpha_1,\gamma_1)\alpha_1+i(\alpha_1,\gamma_1)\alpha_2}
(X_0)}=\frac{\mathrm{Ext}_{\alpha_1}(X_t)}{\mathrm{Ext}_{\alpha_1}(X_0)}
=\frac{\mathrm{Ext}_{\alpha_2}(X_t)}{\mathrm{Ext}_{\alpha_2}(X_0)},$$
which completes the proof of the first part of Proposition \ref{nec}.

To prove Equation (\ref{nec-equ}), by continuity it suffices to show that it holds when $a_1,a_2$ are positive rational numbers.
It remains to show that
$$\{\frac{1}{n}\left(i(\alpha_1,\gamma),i(\alpha_2,\gamma)\right) :
\gamma\in \mathcal{S}, n \in \mathbb{N}\}$$
is dense in $\{(a_1,a_2) : a_1,a_2 \in \mathbb{Q}_+\} $.
We consider the case where  both $\alpha_1,\alpha_2$ are separating simple closed curves;
the case where both $\alpha_1,\alpha_2$ are un-separating can be justified by a similar argument.

Given any $(p_1/ q_1,p_2/q_2)\in \mathbb{Q}_+\times\mathbb{Q}_+$. Since $(2p_1q_2,2p_2q_1)/(2q_1q_2)=(p_1/ q_1,p_2/q_2)$,
it suffices  to show that for any $\epsilon>0$, there exists a $\gamma\in \mathcal{S}$ such that
$$|i(\gamma,\alpha_1)-2p_1q_2|<\epsilon,$$
$$|i(\gamma,\alpha_2)-2p_2q_1|<\epsilon.$$
Note that (by the assumption that  both $\alpha_1,\alpha_2$ are separating) there are $\gamma_1,\gamma_2\in \mathcal{S}$ such that
$i(\gamma_i,\alpha_i)=2,(\gamma_i,\alpha_j)=0\ (i,j\in\{1,2\}, i\neq j).$

Let us denote by $T^n_{\alpha}$ the n-order Dehn twist around $\alpha$.
Then $\gamma'_i=T^{p_iq_{i+1}}_{\gamma_i}(\alpha_i)$ (here $q_3=q_1)$ satisfies the following properties:
\begin{eqnarray*}
i(\gamma'_1,\gamma'_2)&=&0,\\
i(\gamma'_i,\alpha_i)&=&2p_iq_{i+1},\\
i(\gamma'_i,\alpha_j)&=&0, \ \text{if} \ i\neq j.
\end{eqnarray*}
Therefore any simple closed curve $\gamma$ approximating  $\gamma'_1+\gamma'_2$ will be the one we are looking for.
\end{proof}

Define $\nu$ as above. The proof of Proposition \ref{necessary} implies that for any $X\in \mathrm{T}(M)$,
$h_n(X)$ converges to the limit $p$ on the Gardiner-Masur boundary such that

$$\varepsilon_p(\gamma)=\mathrm{Ext}_{i(\alpha_1,\gamma)\alpha_1+i(\alpha_2,\gamma)\alpha_2}
(X)$$
up to a multiplicative constant. Gardiner and Masur \cite{GM} showed that the earthquake path ${E^{t\nu}(X)}$
 cannot converge to any measured lamination on the Gardiner-Masur boundary. We sketch the proof here.

Otherwise, assume that $h_n(X)\rightarrow [F]\in \mathcal{PMF}$. Note that for any $\gamma\in \mathcal{S}$,
$$ \frac{\mathrm{Ext}_{\gamma}(h_n(X))}{n^2}\rightarrow \mathrm{Ext}_{i(\alpha_1,\gamma)\alpha_1+i(\alpha_2,\gamma)\alpha_2}(X).$$
In particular, any $\gamma\in \mathcal{S}$ with $i(\gamma,\alpha_j)=0, j=1,2$ satisfies
 $$\frac{\mathrm{Ext}_{\gamma}(h_n(X))}{n^2}\rightarrow 0.$$
It is not hard to see that $F$ is a measured foliation with non-critical leaves  homotopic to $\alpha_1$ or $\alpha_2$.

Since $$\lim_{n\rightarrow\infty}\sqrt{\frac{\mathrm{Ext}_{\beta_1}(h_n(X))}
{\mathrm{Ext}_{\beta_2}(h_n(X))}}=\sqrt{\frac{\mathrm{Ext}_{\alpha_1}(X)}{\mathrm{Ext}_{\alpha_2}(X)}}
=\frac{i(F,\beta_1)}{i(F,\beta_2)},$$
we can assume that $F=\sqrt{\mathrm{Ext}_{\alpha_1}(X)}\alpha_1+\sqrt{\mathrm{Ext}_{\alpha_2}(X)}\alpha_2$ (i.e.,  the height of the cylinder of $\alpha_i$ is $\sqrt{\mathrm{Ext}_{\alpha_i}(X)}.$

Let $\gamma\in \mathcal{S}$ with $i(\gamma,\alpha_i)=k_i, i=1,2$. Then
$$ \frac{\mathrm{Ext}_{\gamma}(h_n(X))}{n^2}\rightarrow \mathrm{Ext}_{k_1\alpha_1+k_2\alpha_2}(X),$$
and  $$\sqrt{\frac{\mathrm{Ext}_{\gamma}(h_n(X))}{n^2}}\rightarrow i(F,\gamma)=k_1 \sqrt{\mathrm{Ext}_{\alpha_1}(X)}+k_2 \sqrt{\mathrm{Ext}_{\alpha_2}(X)}.$$
However, if we assume that $q$ is a holomorphic quadratic differential on $X$ whose vertical measured foliation is equal to $k_1 \alpha_1+k_2 \alpha_2$
($q$ is called the \emph{Hubbard-Masur differential} of $k_1 \alpha_1+k_2 \alpha_2$),
then
\begin{eqnarray*}
\mathrm{Ext}_{k_1\alpha_1+k_2\alpha_2}(X)&=& \| q \| =k_1 L_q(\alpha_1)+k_2 L_q(\alpha_2) \\
&=&(\frac{k_1L_q(\alpha_1)}{\parallel q\parallel^{\frac{1}{2}}}+\frac{k_2L_q(\alpha_2)}{\parallel q\parallel^{\frac{1}{2}}})^{2} \\
&<& k_1 \sqrt{\mathrm{Ext}_{\alpha_1}(X)}+k_2 \sqrt{\mathrm{Ext}_{\alpha_2}(X)}.
\end{eqnarray*}
This is a contradiction.
(The first equality follows from an result of Kerckhoff \cite{Ker2} saying that the extremal length of a measured foliation
$F$ is equal to the area of the Hubbard-Masur differential whose vertical measured foliation is equivalent to $F$).

\section{The Teichm\"uller horocycle flow}\label{sec:horocycle}

In this section, we investigate the convergence of  the horocycle flow in Teichm\"uller space to the Thurston or  Gardiner-Masur boundary.

\subsection{The Teichm\"uller horocycle flow $h^t$}
Let $\mathcal{QT}(M)$ be the bundle of holomorphic quadratic differentials over the Teichm\"uller space $\mathcal{T}(M)$
and $\pi:\mathcal{QT}(M) \to \mathcal{T}(M)$ be the natural projection.

Recall that any $q\in \mathcal{QT}(M)$ can be defined by a pair of transverse measured foliations on $M$:
 the vertical measured foliation $V_q$  and the horizontal measured foliation $H_q$.
Specifically, in a neighborhood of a
nonsingular point, there are natural coordinates $z = x + iy $ so
that the leaves of $H_q$ are given by $y$ = constant, and
the transverse measure of $H_q$ is $|dy|$. Similarly, the leaves of
$V_q$ are given by $x$ = constant, and the transverse
measure is $|dx|$. The foliations $H_q$ and
$V_q$ have zero set of $q$ as their common singular
set, and at each zero of order $k$ they have a $k+2$-pronged
singularity, locally modeled on the singularity at the origin of
$z^k dz^2$.

 For simplicity, we set $q=(F_1, F_2)$ if $F_1=V_q,\ F_2=H_q$.
  The Teichm\"uller horocycle flow $h^t(q)\in \mathcal{QT}(M)$ is defined as follows. In local coordinates $z=x+iy$ on $M$ in which $q=d z^2$,
  the vertical measured foliation $V(t)$ of $h^t(q)$ is determined by the form $|dx+tdy|$. More precisely, the leaves of $V(t)$ are given by $x+ty=$constant, and the transverse measure is given by $|dx+tdy|$. Note that the horizontal measured foliation $H(t)$ is left invariant (see \cite{Ma2} for more details).

\subsection{Flat metrics}
It is well known that a  holomorphic quadratic differential $q$ induces a flat metric on $M$,
which is $\sqrt{|q|}|dz|$ in local coordinates.
It is easy to see that two quadratic differentials which induce the same flat metric differ only by a rotation.
Accordingly the \emph{space of flat metrics} is defined as
\[\text{Flat}(M)=\mathcal{Q}\mathcal{T}(M)/{\thicksim}, \]
where we set ${q\thicksim e^{i\theta}q}$.

We let
$$\mathcal{Q}^1\mathcal{T}(M)=\{q\in\mathcal{QT}(M):||q||=\iint_{M}|q|dxdy=1\}.$$
For more details on flat metrics, see \cite{DLR} and the references therein.

\subsection{Geodesic currents}
Let $\mathbb{H}^2$ be the Poincar\'e disc and $\partial{\mathbb{H}^2}$ be its ideal boundary,
identified with $S^1$.
 Recall that any two distinct points in $S^1$ determine a complete geodesic on $\mathbb{H}^2$.
Consider the space
$$ G(\mathbb{H}^2)=(S^1\times S^1\backslash {\Delta})/(x,y)\thicksim (y,x)$$
of unoriented geodesics on the Poincar\'e disc.
We fix a hyperbolic structure on $M$ (where $\pi_1(M)$ is identified with the Fuchsian group).
Let us endow  $G(\mathbb{H}^2)$  with the diagonal action of $\pi_1(M)$. A \textit{geodesic current} is a  $\pi_1(M)$-invariant Radon measure on $ G(\mathbb{H}^2)$. Let $C(M)$ be the space of geodesic currents on $M$ equipped with the weak* topology and $\mathcal{P}C(M)=C(M)/\R_+$ the space of projective geodesic currents.

We denote by $\mathcal{C}$ the union of the space of homotopy classes of essential closed curves and the space of homotopy classes of essential proper paths on $M$, where an \emph{essential proper path} is a path $\alpha: (a,b)\to M$ for which $\alpha(t)$ tends to a puncture point as $t\to a\ (or\ b)$.

The followings properties of $C(M)$ will be used in this paper.
\begin{theorem}[Bonahon \cite{Bon2}]
The geometric intersection number $i:\mathcal{ML} \times \mathcal{ML} \to \R$ has a continuous, bilinear extension to
$$ i: C(M) \times C(M) \to \R.$$
\end{theorem}
\begin{theorem}[Duchin, Leininger, Rafi \cite{DLR}]\label{DLR}
A sequence of currents $\mu_k \in C(M)$ converges to $\mu \in C(M)$ if and only if
$$ \lim_{k\to \infty}i(\mu_k,\alpha)=i(\mu,\alpha), \forall \  \alpha \in \mathcal{C}.$$
\end{theorem}
\begin{theorem}[Duchin, Leininger, Rafi \cite{DLR}]\label{DLR}
For any $q\in$ $\mathrm{Flat}(S)$, there exists a current $L_q$ such that  $i(L_q,\alpha)=l_q(\alpha)$ for all $\alpha \in \mathcal{C}$.
\end{theorem}

\subsection{Convergence of horocycle flows }

 Let $\pi: \mathcal{Q}\mathcal{T}(M) \to \mathcal{T}(M)$ be the natural projection. Let $h^t$ be a Teichm\"uller horocycle flow. The following theorem is our main result in this section.
\begin{theorem}\label{compatibility of horocyle} Let  $q\in \mathcal{Q}^1\mathcal{T}(M)$ and
\begin{eqnarray*}
 h:\R& \to& \mathcal{Q}\mathcal{T}(M)\\
    t&\mapsto & h^t(q)
\end{eqnarray*}
be the Teichm\"uller horocycle flow. Then any accumulation point of the horocycle path $\{\pi(h^t(q)\}_{t\in\R}$ on the Thurston boundary $\PMF$ is contained in the set $\{[\mu]\in\PMF: i(\mu,V_q)=0\}.$
\end{theorem}
As an application, we get
 \begin{corollary}
If the vertical measured foliation $V_q$ of $q$ is uniquely ergodic,
then the horocycle path $\{\pi(h^t(q))\}_{t\in\R}$ converges to $[V_q]$ in $\PMF$.
 \end{corollary}
 \begin{proof}
 Since  $V_q$ is uniquely ergodic, the set $\{[\mu]\in\PMF: i(\mu,V_q)=0\}$ contains only one element $[V_q]$, therefore the horocycle path  $\{\pi(h^t(q)\}_{t\in\R}$ converges to $[V_q]$ in $\PMF$.
 \end{proof}
 \begin{corollary}
If the vertical measured foliation $V_q$ of $q$ is uniquely ergodic,
then the horocycle path $\{\pi(h^t(q))\}_{t\in\R}$ converges to $[V_q]$ on the Gardiner-Masur boundary.
\end{corollary}
\begin{proof}
The proof is similar to that of Theorem \ref{sim & uniq}.
\end{proof}
Before proving Theorem \ref{compatibility of horocyle}, we need to do some preparations.
\begin{lemma}\label{h keeps Q^1}
$h^t$ keeps $\mathcal{Q}^1\mathcal{T}(M)$ invariant.
\end{lemma}
The above lemma is well known and we neglect the proof.

\begin{proposition}\label{horocycle convergence in geodesic current}
Assume $q\in \mathcal{Q}^1\mathcal{T}(M)$, and let
\begin{eqnarray*}
 h:\R& \to& \mathcal{Q}^1\mathcal{T}(M)\\
    t&\mapsto & h^t(q).
\end{eqnarray*}
Then, considering $h^t(q)$ as an element of $C(M)$, we have
\[\frac{L_{h^t(q)}}{|t|} \to V_q\]
as $|t|\to\infty$, where $V_q$ is the vertical measured foliation of $q$.
\end{proposition}
\begin{proof}
For any closed curve $\alpha \in \mathcal{C}$, the flat length $l_{h^t(q)}(\alpha)$ of $\alpha$ is less than the sum of its horizontal length $i(V_{h^t(q)},\alpha)$ and its vertical length $i(H_{h^t(q)},\alpha)$ and is larger than the maximum of its horizontal length $i(V_{h^t(q)},\alpha)$ and its vertical length $i(H_{h^t(q)},\alpha)$.

By definition, we have the following:
\begin{eqnarray*}
i(H_{h^t(q)},\alpha)&=&i(H_{q},\alpha),\\
i(V_{h^t(q)},\alpha)&=&\int_{\alpha}|dx+tdy|\\
                    &\leq&\int_{\alpha}|dx|+|t||dy|\\
                    &=&i(H_{q},\alpha)+|t|i(V_{q},\alpha),\\
i(V_{h^t(q)},\alpha)&\geq & -i(H_{q},\alpha)+|t|i(V_{q},\alpha).
\end{eqnarray*}
Therefore
\[\frac{-i(H_{q},\alpha)}{|t|}+i(V_{q},\alpha)\leq \frac{L_{q_t}(\alpha_i)}{|t|}\leq \frac{2i(H_{q},\alpha)}{|t|}+i(V_{q},\alpha).\ \]
This implies
\[ \frac{L_{q_t}(\alpha_i)}{|t|} \to i(V_{q},\alpha)\]
as $ |t|\to\infty.$
Now the proposition follows from  Lemma \ref{DLR}.
\end{proof}

The last ingredient we need to prove Theorem \ref{compatibility of horocyle} is the following lemma.
\begin{lemma}[\cite{DLR}]\label{compatibility}
Let $q_n$ be a sequence of flat structures on $M$ and $\sigma_n=\pi(q_n)$. Assume that $\sigma_n\to\mu$ on the Thurston compactification and $q_n \to \eta$ in $\mathcal{P}C(M)$. Then $i(\mu,\eta)=0$.
\end{lemma}
\begin{proof}[Proof of Theorem \ref{compatibility of horocyle}]
From Proposition \ref{horocycle convergence in geodesic current}, we know
 that $L_{h^t(q)} \to V_q$ in $\mathcal{P}C(M)$. Assume $[\mu]$ is an accumulation point of the horocycle path $\{H^t(q)\}_{t\in\R}$. There is a sequence of positive numbers $\{t_n\}$ such that $\pi(h^t(q))\to \mu $ on the Thurston compactification. Hence $i(\mu,V_q)=0$ by Lemma \ref{compatibility}.
\end{proof}

\section{Concluding remarks}

We have shown that any Teichm\"uller horocycle flow induced by a holomorphic quadratic differential with uniquely ergodic
vertical measured foliation converges to the Gardiner-Masur boundary.
Since the Gardiner-Masur boundary is a natural boundary for the geometry of Teichm\"uller distance \cite{Miy1,Miy2,LS},
it is natural to ask how the orbit of a general horocycle flow behave in the Gardiner-Masur
compactification.

The necessary condition in Proposition \ref{necessary} seems far-fetched. Thus we believe that it is possible to construct rational earthquake flows (that is, earthquakes directed by weighted multicurves)
that are not converge to the Gardiner-Masur boundary.

As we mentioned above,  M. Mirzakhani showed that the Teichm\"uller horocycycle flow is measurably
equivalent to the earthquake flow. Hence, it is also natural to ask how the two flows are different ``geometrically".
We point out that Minsky and Weiss \cite{MW} showed that, in the case of a once-punctured torus or a four-times punctured sphere,
the earthquake flow and horocycle flow, induced by the same measured laminations, may have  infinite Hausdorff (Teichm\"uller) distance.

\end{document}